\documentclass[a4paper, 11pt, reqno]{amsart}

\usepackage{amssymb}

\usepackage{enumitem}
\setlist[enumerate,1]{label=\rm(\arabic*)}
\setlist[enumerate,2]{label=\rm(\alph*)}
\setlist[enumerate,3]{label=\rm(\roman*)}

\usepackage[PS]{diagrams}
\usepackage{xcolor}
\usepackage{pgf}

\newcommand{\ie}{\textit{i.e.}}

\newcommand{\viz}{\textit{viz.}}
\newcommand{\ignore}[1]{\relax}
\newcommand{\bN}{\mathbb{N}}
\newcommand{\bZ}{\mathbb{Z}}

\newcommand{\nbd}{\nobreakdash}

\newcommand{\rr}{\mathrm{r}}

\numberwithin{equation}{section}

\newtheorem{theorem}[equation]{Theorem}
\newtheorem{corollary}[equation]{Corollary}

\newtheorem{lemma}[equation]{Lemma}

\theoremstyle{definition}

\newtheorem{remark}[equation]{Remark}

\newtheorem{example}[equation]{Example}

\let\ideal=\unlhd

\let\oldtocsubsection=\tocsubsection
\renewcommand{\tocsubsection}[2]{\hspace{3.5em}\(\cdot\)~\oldtocsubsection{#1}{#2}}

\usepackage[hypertexnames=false,colorlinks=false,pdfborderstyle={/S/U/W 1}]{hyperref}
\usepackage{bookmark}

\newcommand{\supp}{\mathrm{supp}}

\newcommand{\inv}{^{-1}}

\title{Annihilators in $\bN^{k}$-graded and $\bZ^{k}$-graded rings}

\date{\today}

\author{Thomas H\"uttemann}

\address{Thomas H\"uttemann\\ Queen's University Belfast\\ School of
  Mathematics and Physics\\ Mathematical Sciences Research Centre\\
  Belfast BT7~1NN\\ Northern Ireland, UK}

\email{t.huettemann@qub.ac.uk}

\urladdr{http://www.qub.ac.uk/puremaths/Staff/Thomas\ Huettemann/}

\thanks{Work on this paper commenced during a research visit of the
  author to the Beijing Institute of Technology in January~2017. Their
  hospitality and financial support are gratefully acknowledged.}


\begin{document}

\begin{abstract}
  It has been shown by \textsc{McCoy} that a right ideal of a
  polynomial ring with several indeterminates has a non-trivial
  homogeneous right annihilator of degree~$0$ provided its right
  annihilator is non-trivial to begin with. In this note, it is
  documented that any $\bN^{k}$-graded ring~$R$ has a slightly weaker
  property: the right annihilator of a right ideal contains a
  homogeneous non-zero element, if it is non-trivial to begin with. If
  $R$ is a subring of a $\bZ^{k}$-graded ring~$S$ satisfying a certain
  non-annihilation property (which is the case if $S$ is strongly
  graded, for example), then it is possible to find annihilators of
  degree~$0$.
\end{abstract}

\maketitle

\section{Introduction}

In 1942, \textsc{McCoy} proved the following remarkable result
concerning zero divisors in polynomial rings:

\begin{theorem}[{\textsc{McCoy} \cite[Theorems~2 and~3]{MR0006150}}]
  \label{thm:original_1}
  Let $K$ be a commutative unital ring, and let
  $f \in K[x_{1},\, x_{2},\, \cdots,\, x_{k}]$ be a non-zero
  polynomial which is a zero divisor
  in~$K[x_{1},\, x_{2},\, \cdots,\, x_{k}]$. Then there exists a
  non-zero $c \in K$ such that $f \cdot c = 0$.
\end{theorem}

The result is not valid for non-commutative~$K$. Indeed, consider the
ring $R = \mathrm{Mat}_2 (\bZ)$ of square matrices of size~$2$
over~$\bZ$, and write $E_{ij}$ for the usual matrix units. The
polynomials $f = E_{21} + E_{11} x + E_{22} x^2 + E_{12} x^3$ and
$ g = E_{11} - E_{12} x$ satisfy $g \cdot f = 0$, but no $M \in R$
annihilates~$f$. Nevertheless, an ideal-theoretic variation of
Theorem~\ref{thm:original_1} holds:

\begin{theorem} [{\textsc{McCoy} \cite[Theorem]{MR0082486}}]
  \label{thm:original_ideal}
  Let $K$ be a unital ring. If the right ideal $A$ of
  $K[x_{1},\, x_{2},\, \cdots,\, x_{k}]$ has non-trivial right
  annihilator, then there exists a non-zero $c \in K$ with
  $A \cdot c = \{0\}$.
\end{theorem}

This can be re-phrased in terms of a grading on the polynomial ring:
{\it If the right ideal~$A$ has non-trivial right annihilator, then
  there exists a non-zero annihilator which is homogeneous of
  degree~$0$. In particular, the right annihilator of~$A$ contains a
  non-zero graded ideal of $K[x_{1},\, x_{2},\, \cdots,\, x_{k}]$
  which intersects $K$ non-trivially.}

Let $G$ be an additive monoid with zero element~$0$. We say that a
$G$-graded ring~$T$ has the {\it graded right \textsc{McCoy} property
  for elements\/} if for all non-zero $f,g \in T$ with $f \cdot g = 0$
there exists a homogeneous element $h \in T$ of degree~$0$ such that
$f \cdot h = 0$. We say that $T$~has the {\it graded right
  \textsc{McCoy} property for right ideals\/} if for all right ideals
$A \ideal T$ with non-trivial right annihilator~$A^{\rr}$ there exists
a non-zero homogeneous element of degree~$0$ in~$A^{\rr}$.

Thus a polynomial ring $T = K[X_{1},\, X_{2},\, \cdots,\, X_{k}]$, for
$K$ a unital ring, has the graded right \textsc{McCoy} property for
right ideals by Theorem~\ref{thm:original_ideal}, and if $K$ is unital
and commutative, it has the graded right \textsc{McCoy} property for
elements by Theorem~\ref{thm:original_1}.

It is the purpose of this note to re-visit these results strictly from
the point of view of graded algebra, and provide a different
perspective on the special place polynomial rings occupy in the
theory. We show that every $\bN^{k}$-graded ring~$T$ has the {\it weak
  graded right \textsc{McCoy} property for right ideals\/}
(Theorem~\ref{thm:Nk}): if the right ideal $A \ideal T$ has
non-trivial right annihilator~$A^{\rr}$ there exists a non-zero
homogeneous element, of possibly non-zero degree, in~$A^{\rr}$. If $T$
arises as the $\bN^{k}$-graded subring of a {\it strongly\/}
$\bZ^{k}$-graded ring, or more generally of a $\bZ^{k}$-graded ring
satisfying a certain non-annihilation condition, then $T$ actually
possesses the graded right \textsc{McCoy} property for right ideals
(Theorem~\ref{thm:half-strong}); this applies, for example, to
polynomial rings.

For semi-commutative~$T$ the (weak) graded \textsc{McCoy} property for
right ideals implies the (weak) graded \textsc{McCoy} property for
elements, which is recorded in
Corollaries~\ref{cor:weak_McCoy_elements} and~\ref{cor:McCoy-again}.

\section{Notation and conventions}

Given a right ideal~$A$ of a (possibly non-unital) ring~$R$ we write
$A^{\rr}$ for the right annihilator of~$A$ in~$R$, that is,
\begin{displaymath}
A^{\rr} = \{ r \in R \,|\, \forall a \in A \colon ar = 0 \} \ .
\end{displaymath}
The set~$A^{\rr}$ is a two-sided ideal of~$R$.

\subsection*{Rings graded by a monoid}

Given a monoid~$G$, additively written, a $G$-graded ring is a
ring~$R$ together with a decomposition $R = \bigoplus_{g \in G} R_{g}$
into \textsc{abel}ian groups such that $R_{g} R_{h} \subseteq R_{g+h}$
for all $g,h \in G$. Elements of~$R_{g}$ are called {\it homogeneous
  of degree~$g$}; we may say {\it $R$-homogeneous of degree~$g$\/} if
we want to emphasise the ring~$R$ and its grading.

Every element $r$ of a $G$-graded ring~$R$ can be uniquely written as
a sum
\begin{equation}
  \label{eq:0}
  r = \sum_{g \in G} r_{g}
\end{equation}
where $r_{g} \in R_{g}$, with almost all $r_{g}$ zero. The set
$\supp(r) = \{g \in G \,|\, r_{g} \neq 0\}$ is the {\it
  support of~$r$}. --- The following elementary Lemma is central:

\begin{lemma}
  \label{lem:cancellative'}
  Suppose $G$ is an additively written monoid with neutral element~$0$
  (so that $a+0 = a = 0+a$ for all $a \in G$). Let $T$ be a $G$-graded
  ring.
  Let $r = \sum_{g \in G} r_{g} \in T$ as in~\eqref{eq:0}, and let
  $s \in T$ be homogeneous. 
    \begin{enumerate}
    \item If $G$ is right cancellative so that $a+c = b+c$ implies
      $a=b$, then $rs = 0$ implies $r_{g} s = 0$ for all $g \in
      G$. \label{item:2a}
    \item If $G$ is left cancellative so that $c+a = c+b$ implies
      $a=b$, then $sr = 0$ implies $s r_{g} = 0$ for all $g \in
      G$. \label{item:2b}
    \end{enumerate}
\end{lemma}

\begin{proof}
  We prove~(1) only. The elements $r_{g} s$ is homogeneous of
  degree~$g+h$. Now $g+h = g'+h$ if and only if $g=g'$ as $G$ is right
  cancellative, so $rs = \sum_{g \in G} r_{g} s$ is the unique
  decomposition of~$rs$ into homogeneous elements of distinct
  degree. Thus $rs = 0$ entails $r_{g} s = 0$ for all $g \in G$ as
  claimed.
\end{proof}

One can also show that {\it if $G$ is right cancellative or left
  cancellative, and if $T$ has a unity, then the unit element
  $1 \in T$ is homogeneous of degree~$0$}. Indeed, supposing that $G$
is right cancellative write $r = 1 = \sum_{g \in G} r_{g}$ as
in~\eqref{eq:0}. For any $h \in G$ and any homogeneous element
$s \in T_{h}$ we find $s = 1 \cdot s = \sum_{g \in G} r_{g} s$, with
$r_{g} s \in T_{g+h}$.  By uniqueness of the representation, this
implies $r_{g} s = 0$ whenever $h \neq g+h$, \ie, whenever $g \neq 0$
(recall that $G$ is right cancellative). By distributivity, we have
$r_{g} z = 0$ for any $z \in T$ and $g \neq 0$. Applying this
to~$z=1$ yields
$1 = z = 1 \cdot z = \sum_{g \in G} r_{g} z = r_{0} z = r_{0} \cdot 1
= r_{0}$
which shows that the unit is homogeneous of degree~$0$ as claimed. ---
For $G$ a group this is an observation of \textsc{Dade}
\cite[Proposition~1.4]{GRD}.

\subsection*{Rings graded by~$\bN$}

Suppose that $T = \bigoplus_{j \geq 0} T_{j}$ is an $\bN$-graded
ring. Every non-zero element $z \in T$ can be written uniquely in
the form
\begin{equation}
  \label{eq:1}
  z = \sum_{j=\ell}^{u} z_j
\end{equation}
with $z_{j} \in T_{j}$, and both $z_{\ell}$ and~$z_{u}$ non-zero.  We
call the expression~\eqref{eq:1} the {\it canonical form of~$z$}. We
say that $z$ has {\it lower degree\/}~$\ell$ and {\it upper
  degree\/}~$u$; the quantity $a = u - \ell \geq 0$ is called the {\it
  amplitude\/} of~$z$. The element $z$ is homogeneous if and only if
it is of amplitude~$0$. We remark that if $z$ has amplitude~$a$ and
$x$ is a homogeneous element, then $zx$ has amplitude not
exceeding~$a$.  Indeed, if $\ell$ and~$u$ denote the lower and upper
degree of~$z$, respectively, and $x$ has degree~$d$, then the lower
degree of $zx$ is at least $\ell +d$ while the upper degree is at
most~$u+d$. Inequality occurs if and only if $z_{\ell} x = 0$ or
$z_{u} x = 0$.

\section{Annihilators in $\bN$-graded rings}

\begin{theorem}
  \label{thm:amplitude}
  Suppose that $T$ is an $\bN$-graded ring. Let $A$ be a right ideal
  in~$T$ such that $A^{\rr}$, its right annihilator, is non-zero. If
  $A^{\rr}$ contains a non-zero element~$y$ of positive amplitude~$a$,
  then $A^{\rr}$ also contains a non-zero element~$z$ of amplitude
  less than~$a$. In particular, $A^{\rr}$ contains a non-zero
  homogeneous element.
\end{theorem}

A version of this Theorem for commutative polynomial rings, referring
to degree rather than amplitude, was given by \textsc{Forsythe}
\cite[Theorem~A]{MR0007393}.

\begin{proof}
  Write $y$ in canonical form $y = \sum_{\ell}^{u} y_j$; note that
  $a = u - \ell > 0$ by hypothesis. If $y_{u} \in A^{\rr}$ then
  $z = y_{u}$ has amplitude~$0$, and is the desired element of
  amplitude less than~$a$.

  Otherwise, if $y_{u} \notin A^{\rr}$, we can choose an element
  $x \in A$, with canonical form $\sum_{i=m}^{n} x_{i}$ , such that
  $xy_{u} \neq 0$. If $x_{i} y = 0$ for all~$i$ then we must in
  particular have $x_{i}y_{u} = 0$ for all~$i$ by
  Lemma~\ref{lem:cancellative'}~\ref{item:2b}. Thus
  $x y_{u} = \sum_{m}^{n} x_{i} y_{u} = 0$ contradicting the choice
  of~$x$. Consequently, there exists a maximal index~$p$ with
  $x_{p} y \neq 0$. But $y \in A^{\rr}$ so that
  \begin{displaymath}
    0 = xy = \sum_{i=m}^{n} x_{i} y = \sum_{i=m}^{p} x_{i} y \ .
  \end{displaymath}
  It follows that $x_{p} y_{u} = 0$ so that $z = x_{p} y$ has
  amplitude less than~$a$; indeed, the upper degree of~$x_{p} y$ is
  less than~$p+u$, while the lower degree of $x_{p} y$ is at least
  $p + \ell$. --- As $A^{\rr}$ is a two-sided ideal,
  $z = x_{p} y \in A^{\rr}$ is the desired non-zero element of amplitude
  less than~$a$.

  The last sentence of the Theorem follows by repeated application of
  what we proved already, yielding non-zero elements in~$A^{\rr}$ of
  successively smaller amplitude. The process must stop with an
  element of amplitude~$0$.
\end{proof}

\begin{corollary}
  Suppose that $T$ is an $\bN$-graded ring. Let $A$ be a right ideal
  in~$T$ such that $A^{\rr}$, its right annihilator, is non-zero. Then
  $A^{\rr}$ contains a non-trivial graded ideal of~$T$. \qed
\end{corollary}

The result is best possible. For let $T = K[x]/\langle x^2\rangle$,
with $K$ a field, graded by $\deg(x) = 1$. Let $A = \langle x \rangle$
be the right ideal of all polynomials in~$x$ without constant term. It
is annihilated by $x \in T_{1}$, but there is no non-zero homogeneous
annihilator of degree~$0$.

\begin{remark}
  It is no coincidence that the proof of Theorem~\ref{thm:amplitude}
  makes use of the natural ordering on~$\bN$: One cannot expect
  \textsc{McCoy}-type results unless the grading monoid lies in an
  ordered group. For example, let $K$ be a field and consider
  $R = K[x]/(x^2-1)$ as a $\bZ/2$-graded ring, with $x$ having
  degree~$1$. This ring contains zero divisors as $(1-x)(1+x) = 0$,
  but $R$ does not contain any non-zero {\it homogeneous\/} zero
  divisors. On the other hand, if $T$ is a $G$-graded ring with $G$ an
  ordered group, then any equality $a \cdot b = 0$ with non-zero $a$
  and~$b$ yields, by passing to leading terms with respect to the
  total order on~$G$, to a pair $(x, y)$ of {\it homogeneous\/}
  non-zero elements with $x \cdot y = 0$.
\end{remark}

\section{Annihilators in $\bN^{k}$-graded rings}

\begin{theorem}
  \label{thm:Nk}
  Suppose that $R$ is an $\mathbb{N}^k$-graded ring. Let $A$ be a
  right ideal in~$R$ such that $A^{\rr}$, its right annihilator, is
  non-trivial. Then $A^{\rr}$ contains a non-zero {\it homogeneous\/}
  element.
\end{theorem}

\begin{proof}
  We use induction on~$k$, the case $k=1$ being the final sentence of
  Theorem~\ref{thm:amplitude} applied to $T = R$.

  So suppose that $R$ is $\bN^{k+1}$-graded. We let $T$ denote the
  $\bN$-graded ring which is identical to~$R$ as a ring, but with
  grading defined by the last coordinate of~$\bN^{k+1}$. More
  explicitly, denote by $\tau \colon \bN^{k+1} \rTo \bN$ the
  projection
  \begin{displaymath}
    \bN^{k+1} = \bN^{k} \oplus \bN \rTo \bN \ ;
  \end{displaymath}
  writing $R = \bigoplus_{v \in \bN^{k+1}} R_v$ we let $T =
  \bigoplus_{j=0}^{\infty} T_{j}$ where
  \begin{displaymath}
    T_{j} = \bigoplus_{v \in \tau\inv(j)} R_{v} \ .
  \end{displaymath}
  Now suppose $A$ is a right ideal of the ring $R=T$ which has
  non-trivial right annihilator~$A^{\rr}$. By Theorem~\ref{thm:amplitude}
  we find a non-zero element $y \in A^{\rr}$ which is homogeneous of
  degree~$d$ as an element of~$T$, \ie, $y \in A^{\rr} \cap T_{d}$.

  Write a general element $x \in A \subseteq T$ in canonical form with
  respect to the ring~$T$ (\ie, with respect to the $\bN$-grading),
  \begin{displaymath}
    x = \sum_{i=m}^{n} x_{i} \ ,
  \end{displaymath}
  and let $J$ denote the right ideal of $T=R$ generated by all the
  resulting elements~$x_{i}$ (letting $x$ vary over all of~$A$); that
  is, $J$ is the smallest {\it graded\/} right ideal of~$T$
  containing~$A$.  As $x_{i} \in T_{i}$ for all~$i$, and as
  $y \in T_{d}$ is homogeneous, the equality $xy = 0$ (true since
  $y \in A^{\rr}$) implies $x_{i} y = 0$ for all~$i$, by
  Lemma~\ref{lem:cancellative'}~\ref{item:2a}. Thus in fact
  $y \neq 0$ is a right annihilator of~$J$, that is, $y \in J^{\rr}$

  Next, we let $S$ denote the $\bN^{k}$-graded ring which is identical
  to~$R$ (and~$T$) as a ring, but with grading given by the first $k$
  coordinates of~$\bN^{k+1} = \bN^{k} \oplus \bN$. More explicitly,
  denote by $\sigma \colon \bN^{k+1} \rTo \bN^{k}$ the projection
  \begin{displaymath}
    \bN^{k+1} = \bN^{k} \oplus \bN \rTo \bN^{k} \ ;
  \end{displaymath}
  writing $R = \bigoplus_{v \in \bN^{k+1}} R_v$ as before we let
  $S = \bigoplus_{s \in \bN^{k}} S_{s}$ where
  \begin{displaymath}
    S_{s} = \bigoplus_{v \in \sigma\inv(s)} R_{v} \ .
  \end{displaymath}
  Recall now that $J \ideal S$ has non-trivial right annihilator as it
  contains the element $y \neq 0$ constructed above. By our induction
  hypothesis, applied to the $\bN^{k}$-graded ring~$S$ and the
  ideal~$J$, we can find a homogeneous non-zero right annihilator~$z$
  of~$J$ in~$S$, of degree $s \in \bN^{k}$ say. Such an element can
  uniquely be written as a sum of non-zero elements
  \begin{displaymath}
    z = z_{v_{1}} + z_{v_{2}} + \ldots + z_{v_{\ell}} \ ,
  \end{displaymath}
  with $z_{v_j} \in R_{v_j}$ and $\sigma(v_j) = s$ for all~$j$ such
  that
  \begin{displaymath}
    \tau(v_{1}) < \tau(v_{2}) < \ldots < \tau(v_{\ell}) \ .
  \end{displaymath}
  We specifically choose $z$ with $\ell$ as small as possible.  If
  $\ell = 1$ we are done: the element $z = z_{v_1}$ is a homogeneous
  element of~$R$ which annihilates~$J$, and thus annihilates~$A$; note
  that $J^{\rr} \subseteq A^{\rr}$. On the other hand, $\ell > 1$ cannot
  happen. Indeed, if $\ell > 1$ then $z_{v_{\ell}} \notin J^{\rr}$ by
  minimality of~$\ell$. This means we can find an element $x \in J$
  with $x z_{v_{\ell}} \neq 0$. In fact, as $J \ideal T$ is a graded
  ideal, as remarked before, we can ensure that $x$ is {\it
    $T$-homogeneous}. But then $x z = 0$ implies $x z_{v_{\ell}} = 0$
  by Lemma~\ref{lem:cancellative'}~\ref{item:2b}, a
  contradiction. Thus we {\it must\/} have $\ell =1$, finishing the
  induction.
\end{proof}

\begin{corollary}
  Suppose that $R$ is an $\bN^{k}$-graded ring. Let $A$ be a right
  ideal in~$R$ such that $A^{\rr}$, its right annihilator, is
  non-zero. Then $A^{\rr}$ contains a non-trivial graded ideal
  of~$R$. \qed
\end{corollary}

\begin{corollary}
\label{cor:weak_McCoy_elements}
  Let $R$ be an $\bN^{k}$-graded ring, and let $(f_{i})_{i \in I}$ be
  a family of elements of~$R$. Suppose there exists a non-zero element
  $g \in R$ such that $f_{i} g = 0$ for all~$i$. Suppose further that
  $R$ is semi-commutative so that $ab=0$ implies $arb=0$ for all
  $r \in R$.  Then there exists a non-zero homogeneous element
  $h \in R_{v}$, of possibly non-zero degree~$v \in \bN^{k}$, such
  that $f_{i} h = 0$ for all $i \in I$.
\end{corollary}

\begin{proof}
  Let $A = \langle f_{i} \, | \, i \in I \rangle$ be the right ideal
  generated by the elements~$f_{i}$. As $R$ is semi-commutative, $g$
  is a non-trivial annihilator of~$A$. Hence by Theorem~\ref{thm:Nk}
  there exists a homogeneous non-zero element $h \in A^{\rr}$; this
  element~$h$ annihilates in particular the specified generators~$f_i$
  of~$A$.
\end{proof}

\section{Positive subrings of $\bZ^{k}$-graded rings}

Let $S = \bigoplus_{v \in \bZ^{k}} S_{v}$ now denote a
$\bZ^{k}$-graded ring. We will consider the following condition, for
various elements $v \in \bZ^{k}$:
\begin{equation}
  \label{eq:2'}
  \text{For \(x \in S_{v}\), if \(x \neq 0\) then \(x S_{-v} \neq
    \{0\} \).} 
\end{equation}
This is equivalent to saying that the left annihilator of~$S_{-v}$ has
trivial intersection with~$S_{v}$.

The ring $S$ admits an $S_0$-valued ``inner product''
$\langle a,b\rangle = (ab)_{0}$, 
the degree\nbd-$0$ component of the product~$ab$. It is called {\it
  right non-degenerate\/} if $\langle a, S \rangle = \{0\}$ implies
$a=0$, for all $a \in S$. Rings with non-degenerate inner product were
investigated by \textsc{Cohen} and \textsc{Rowen} \cite{MR696990}.

\begin{lemma}
  \label{lem:good_rings}
  \begin{enumerate}
  \item The inner product $\langle \,\cdot\, , \,\cdot\, \rangle$
    on~$S$ is right non-degenerate if and only if
    condition~\eqref{eq:2'} holds for every $v \in \bZ^{n}$.
  \item If $S$ is a strongly graded unital ring, the inner product
    $\langle \,\cdot\, , \,\cdot\, \rangle$ on~$S$ is right
    non-degenerate.
  \end{enumerate}
\end{lemma}

\begin{proof}
  We prove~(1) first. Suppose that
  $\langle \,\cdot\, , \,\cdot\, \rangle$ is right non-degenerate, and
  let $x \in S_{v}$ be given. Then
  $x S_{-v} = \langle x, S \rangle = \{0\}$ implies $x=0$, which means
  that condition~\eqref{eq:2'} holds. Conversely,
  suppose~\eqref{eq:2'} holds for all~$v$, and let
  $a = \sum_{w} a_{w} \in S$ (with $a_{w} \in S_{w}$) be such that
  $\langle a, S \rangle = \{0\}$. Then, for any $w \in \bZ^{k}$,
  \begin{displaymath}
    a_{w} S_{-w} = \langle a, S_{-w} \rangle \subseteq \langle a, S
    \rangle = \{0\}
  \end{displaymath}
  so that $a_{w} = 0$ and thus $a = \sum_{w} a_{w} = 0$ as well. This
  shows that $\langle \,\cdot\, , \,\cdot\, \rangle$ is right
  non-degenerate. 

  To prove~(2), let $a = \sum_{w} a_{w} \in S$ (with
  $a_{w} \in S_{w}$) be such that $\langle a, S \rangle = \{0\}$. For
  any $v \in \bZ^{k}$ we have $S_{-v} S_{v} = S_{0}$, by definition of
  strong grading, hence we can find finitely many elements
  $y_{j} \in S_{-v}$ and $z_{j} \in S_{v}$ such that
  $1 = \sum_{j} y_{j} z_{j}$. Now
  $a_{v} y_{j} = \langle a, y_{j} \rangle \in \langle a, S \rangle =
  \{0\}$ so that $a_{v} y_{j} = 0$ for all~$j$. It follows that
  \begin{displaymath}
    a_{v} = a_{v} \cdot 1 = \sum_{j} (a_{v} y_{j}) z_{j} = 0
  \end{displaymath}
  and hence that $a = \sum_{v} a_{v} = 0$.
\end{proof}

\begin{example}[{\textsc{Cohen}-\textsc{Rowen} \cite{MR696990}, Example~3}]
  Let $K$ be a (possibly non-unital) ring with trivial left annihilator
  so that for all $x \neq 0$ there exists $y \in K$ with $xy \neq 0$,
  and let $S$ denote the ring of square matrices of size~$n$ with
  entries in~$K$. We equip $S$ with a $\bZ$-grading by setting
  \begin{displaymath}
    S_{t} = \bigoplus_{j} K \cdot e_{j, j+t} \ ,
  \end{displaymath}
  where $e_{i,j}$ is a formal matrix unit. In particular, $S_{0}$
  corresponds to the main diagonal and $S_{1}$~to the first
  superdiagonal. Given a non-zero element
  $x = \sum_{j} \lambda_{j} e_{j,j+v} \in S_{v}$ there is an index~$p$
  such that $\lambda_{p} \neq 0$. By hypothesis on~$K$ we can choose
  $\mu \in K$ with $\lambda_{p} \mu \neq 0$. Then
  $y = \mu e_{j+v,j} \in S_{-v}$ satisfies
  $xy = \lambda_{p} \mu e_{j,j} \neq 0$. This shows that $S$
  satisfies~\eqref{eq:2'} for all $v \in \bZ$. As $S_{v} = \{0\}$ for
  $|v| \geq n$, the ring~$S$ is not strongly graded.
\end{example}

As a matter of terminology, $S_{+} = \bigoplus_{v \in \bN^{k}} S_{v}$
is the {\it positive subring\/} of~$S$.

\begin{theorem}
  \label{thm:half-strong}
  Suppose that the $\mathbb{N}^{k}$-graded ring $R= S_{+}$ is the
  positive subring of a $\bZ^{k}$-graded ring~$S$, and suppose that
  $S$ satisfies condition~\eqref{eq:2'} for all
  $v \in \bN^{k} \setminus \{0\}$.  Let $A$ be a right ideal in~$R$
  such that $A^{\rr}$, its right annihilator in~$R$, is
  non-trivial. Then $A^{\rr} \cap R_{0} \neq \{0\}$, \ie, $A^{\rr}$
  contains a non-zero homogeneous element~$h$ of degree~$0$.
\end{theorem}

\begin{proof}
  By Theorem~\ref{thm:Nk} there exist $v \in \bN^{k}$ and a non-zero
  element $g \in R_{v} = S_{v}$ such that $g \in A^{\rr}$. If $v = 0$
  we are done. Otherwise, the hypothesis on~$S$ guarantees that there
  exists an element $y \in S_{-v}$ with $gy \neq 0$.  Set $h = g y$;
  by construction $h \in S_{0} = R_{0}$, and $h \in A^{\rr}$ as
  $A^{\rr}$ is a (right) ideal.
\end{proof}

\begin{corollary}
  In the situation of Theorem~\ref{thm:half-strong}, the
  ideal~$A^{\rr}$ contains a graded ideal of~$R$ intersecting $R_{0}$
  non-trivially. \qed
\end{corollary}

Theorem~\ref{thm:half-strong} applies to the $\bN^{k}$-graded ring
$R = K[X_{1},\, X_{2},\, \cdots,\, X_{k}]$ of polynomials with
coefficients in a ring~$K$ with trivial left annihilator (that is,
$xK=\{0\}$ implies $x=0$). Indeed, we have $R = S_{+}$ where $S$ is
the \textsc{Laurent} polynomial ring
\begin{displaymath}
  S = K[X_{1},\, X_{1}^{-1},\, X_{2},\, X_{2}^{-1},\, \cdots,\,
  X_{k},\, X_{k}^{-1}]
\end{displaymath}
equipped with the usual $\bZ^{k}$-grading, giving the
indeterminate~$X_j$ degree~$e_j$, the $j$th unit vector. For
unital~$K$ we recover the classical result of \textsc{McCoy}
\cite[Theorem]{MR0082486}.

\medbreak

To go any further, we need to put stronger conditions on our rings:

\begin{theorem}
  \label{thm:ann_strongly_graded}
  Let $S$ be a strongly $\bZ^{k}$-graded unital ring, and let
  $R = S_{+}$ be its positive subring. Let
  $f_{1},\, f_{2},\, \cdots,\, f_{\ell} \in S$, and suppose that there
  exists a non-zero element $g \in S$ such that
  \begin{equation}
    f_{j} r g = 0 \quad \text{for all~\(j\) and all \(r \in
      R\).} \label{eq:3}
  \end{equation}
  Then there exists a non-zero element $h \in R_{0}$, homogeneous of
  degree~$0$, such that $f_{j}rh = 0$ for all~$j$ and all $r \in R$.
\end{theorem}

\begin{proof}
  Without loss of generality we may assume that $g \in R$. Indeed, we
  may choose a vector $w \in \bN^{k}$, with sufficiently large
  positive entries, such that
  $w + \supp(g) = \{ w + x \,|\, x \in \supp(g) \}\subset \bN^{k}$.
  As $S$ is strongly graded there are finitely many elements
  $x_{i} \in S_{w}$ and $y_{i} \in S_{-w}$ such that
  $\sum_{i} x_{i} y_{i} = 1$. As
  $g = g \cdot 1 = \sum_{i} (g x_{i}) \cdot y_{i}$ is non-zero there
  is an index~$p$ with $gx_{p} \neq 0$. By construction
  $g x_{p} \in R$ and $f_{j}r (g x_{p}) = (f_{j} r g) x_{p} = 0$ for
  all~$j$, so we can replace $g$ with~$gx_{p} \in R$ in~\eqref{eq:3}.

  Similarly, we may choose a $v \in \bN^{k}$ such that
  $v + \supp(f_{j}) \subset \bN^{k}$ for all~$j$.  As $S$ is strongly
  graded there are finitely many elements $x_{i} \in S_{-v}$ and
  $y_{i} \in S_{v}$ such that $\sum_{i} x_{i} y_{i} = 1$. Then
  $y_{i} f_{j} \in R$ by choice of~$v$, and for all indices $i$
  and~$j$, and all $r \in R$, we have
  \begin{displaymath}
    (y_{i} f_{j}) r g = y_{i} (f_{j} r g) = y_{i} \cdot 0 = 0
  \end{displaymath}
  by our hypotheses on~$g$ and the~$f_{j}$. This means that the right
  ideal $A = \big\langle \{ y_{i} f_{j}\} \big\rangle$ of~$R$
  generated by the elements $y_{i} f_{j}$ has a non-trivial right
  annihilator in~$R$, \viz, the element~$g$. By
  Theorem~\ref{thm:half-strong} there exists a non-zero element
  $h \in R_{0}$ annihilating $A$ from the right. In particular,
  $y_{i} f_{j} r h = 0$ for all indices $i$ and~$j$, and all
  $r \in R$. But then we also have the equality
  \begin{displaymath}
    f_{j} r h = \sum_{i} x_{i} (y_{i} f_{j} r h) = 0
  \end{displaymath}
  for all~$j$ and all $r \in R$, proving the Theorem.
\end{proof}

\begin{corollary}
  \label{cor:McCoy-again}
  Let $S$ be a strongly $\bZ^{k}$-graded unital ring. Suppose that $S$
  is semi-commutative so that $ab=0$ implies $asb=0$ for all
  $s \in S$.  Let $f_{1},\, f_{2},\, \cdots,\, f_{\ell} \in S$, and
  suppose that there exists a non-zero element $g \in S$ such that
  $f_{j} g = 0$ for all~$j$. Then there exists a non-zero element
  $h \in S_{0}$, homogeneous of degree~$0$, such that $f_{j}h = 0$ for
  all~$j$.
\end{corollary}

\begin{proof}
  By definition of semi-commutativity the condition $f_{j} g = 0$
  implies $f_{j} s g = 0$ for all $s \in S$. In particular,
  condition~\eqref{eq:3} is satisfied, hence
  Theorem~\ref{thm:ann_strongly_graded} applies.
\end{proof}

\raggedright


\end{document}